\documentclass[reqno, 11pt]{amsart}

 \usepackage{amssymb}
 \usepackage{amsthm}
 \usepackage{amsmath}
 \usepackage{times}
 \usepackage{latexsym}
 \usepackage{xcolor}
 \usepackage{soul}
 \usepackage[mathscr]{eucal}

 \numberwithin{equation}{section}
 
 \newtheorem{theorem}{Theorem}[section]
 \newtheorem{proposition}[theorem]{Proposition}
 
 \newtheorem{corollary}[theorem]{Corollary}

 \title[Remark on quasi Sasakian structures]{Remark on quasi Sasakian structures}
 \author[Emmanuel Gnandi, Fortun\'e Massamba ]{Emmanuel Gnandi*,  Fortun\'e Massamba**}
 \newcommand{\acr}{\newline\indent}
 \address{\llap{*\,}INSA de Toulouse\acr
 	Département de G\'enie Mathématique\acr
 	Université de Toulouse\acr
 	135 avenue de Rangueil\acr
 	31077 Toulouse cedex 4\acr
 	 France} 
 \email{kpanteemmanuel@gmail.com, gnandi@insa-toulouse.fr}
 \thanks{}
 \address{\llap{**\,}Discipline of Mathematics\acr 
 	School of Agriculture and Science\acr
 	University of KwaZulu-Natal\acr
 	Private Bag X01, Scottsville 3209\acr 
 	South Africa}  
 \email{massfort@yahoo.fr, Massamba@ukzn.ac.za}  
 \thanks{}  
 
 \dedicatory{}
 
\begin{document}
 	\begin{abstract}  
 	In this work, we revisit quasi-Sasakian geometry in dimension three and examine how these structures interact with the foliation generated by the Reeb vector field and its basic cohomology. Through a deformation-based approach, we show that a closed, orientable $3$-manifold admits a quasi-Sasakian structure precisely when it is either Sasakian or arises as a K\"ahler mapping torus. In particular, every quasi-Sasakian structure in this setting can be deformed into a Sasakian or a co-K\"ahler one. This result leads to a complete classification of quasi-Sasakian manifolds in dimension three and highlights the geometric and topological features that distinguish the two cases.
 	\end{abstract} 

 	\maketitle

 	\section{Introduction}

 	Quasi-Sasakian manifolds arise as an important generalization of Sasakian and co-K\"ahler geometry within the broader theory of almost contact metric structures. Given an almost contact metric structure $(\phi,\xi,\eta,g)$ on a differentiable manifold $M$, the structure is called \emph{quasi-Sasakian} when it is normal and the associated second fundamental form $\Omega$ is closed. This notion, introduced by Blair in \cite{blair1967quasi}, simultaneously extends the classical Sasakian case  where $d\eta = \Omega$ and the co-K\"ahler case where $d\eta = 0$. Moreover, quasi-Sasakian structures are known to have necessarily odd rank \cite{blair1966thesis}, which places strong geometric restrictions on the underlying manifold.
 	
 	The study of quasi-Sasakian geometry has developed significantly over the past decades, motivated both by intrinsic geometric interest and by applications in CR-geometry and theoretical physics, particularly in magnetic theory and supergravity (see \cite{calin1992contact, friedrich2001almost, rustanov1994geometry} for more details and references therein). Important contributions to the subject include the works of Blair \cite{blair1966thesis, blair1967quasi}, Kanemaki \cite{kanemaki1977quasi, kanemaki1984quasi}, Oubi\~na \cite{oubina1985new}, Chinea and Gonz\'alez \cite{gonzalezchinea1989quasi}, Tanno \cite{tanno1971quasi}, Kirichenko and Rustanov \cite{kirichenko2002geometry}, and Olszak \cite{olszak1982curvature}. In particular, the $3$-dimensional case has received special attention in the literature, for instance in \cite{olszak1986normal, olszak1996conformally}, where the geometry is closely tied to the behavior of the Reeb vector field and the foliation it generates.
 	
 	In dimension $3$, a quasi-Sasakian structure determines a natural decomposition of the tangent bundle into the Reeb direction and a transverse K\"ahler plane distribution. This viewpoint leads naturally to the study of the basic cohomology associated with the Reeb foliation, which encodes essential global information about the manifold. Understanding how the transverse Kähler geometry interacts with the Reeb flow plays a key role in distinguishing Sasakian and co-K\"ahler situations.
 	
 	The objective of this work is to study quasi-Sasakian structures on closed, oriented $3$-dimensional manifolds and to obtain a classification in terms of Sasakian manifolds and Kähler mapping tori. Our approach makes use of basic cohomology techniques together with deformation arguments that allow us to pass from quasi-Sasakian structures to contact or cosymplectic structures in a controlled manner.
 	
 	This paper is organized as follows. In Section~\ref{sec:basics}, we review the fundamental concepts of almost contact metric, cosymplectic, and quasi-Sasakian geometry, placing particular emphasis on the $3$-dimensional case and the associated Reeb foliation. Section~\ref{sec:Result} contains the main results and their proofs, together with the geometric implications of the classification theorem.

 	\section{Quasi Sasakian structure: Basic Notions}\label{sec:basics}
 	
 	This section introduces the fundamental structure studied throughout this work. For a comprehensive background and a collection of illustrative examples, we refer the reader to the following sources: \cite{blair1967quasi, blair2010riemannian,  cappelletti2008$3$, cappelletti2009geometry, desarkar20093, kanemaki1977quasi, kirichenko2002geometry,  olszak1982curvature,   olszak1986normal, olszak1996conformally, rustanov1994geometry, tanno1971quasi} .
 	
 	Let $M$ be a $(2n + 1)$-dimensional manifold endowed with an almost contact structure $(\phi, \xi, \eta)$, i.e., $\phi$ is a tensor field of type $(1, 1)$, $\xi$ is a vector field called Reeb vector field, and $\eta$ is a $1$-form satisfying (see \cite{blair2010riemannian}, for more details)
 	\begin{equation} \label{eq:almost_contact}
 		\phi^2 = -\mathbb{I} + \eta \otimes \xi, \quad \eta(\xi) = 1.
 	\end{equation}   
 	It follows that $\phi\xi = 0$, $\eta \circ \phi = 0$ and $\operatorname{rank}(\phi) = 2n$. Then $(\phi, \xi, \eta, g)$ is called an almost contact metric structure on $M$ if $(\phi, \xi, \eta)$ is an almost contact structure on $M$ and $g$ is a Riemannian metric on $M$ such that (see \cite{blair2010riemannian})   
 	\begin{equation} \label{eq:metric_condition}
 		g(\phi X, \phi Y) = g(X, Y) - \eta(X)\eta(Y),
 	\end{equation}   
 	for any vector fields $X$, $Y$ on $M$.
 	
 	It follows that the $(1, 1)$-tensor field $\phi$ is skew-symmetric and $\eta(\cdot) = g(\xi, \cdot)$. The second fundamental form of $M$ is defined by   
 	\begin{equation} \label{eq:fundamental_form}
 		\Omega(X, Y) = g(X, \phi Y).
 	\end{equation}   	
 	We say that the $1$-form $\eta$ has rank $s=2p$ if $(d\eta)^{p}\ne0$ and $\eta\wedge(d\eta)^{p}\ne0$ and has rank $s=2p+1$ if $\eta\wedge(d\eta)^{p}\ne0$ and $(d\eta)^{p+1}=0$. We say that $s$ is the rank of the almost contact metric structure (see \cite{blair1967quasi, blair2010riemannian}).
 	
 	In what follows, we define the subsequent structures in accordance with the presentations given in \cite{bazzoni2015k, cappelletti2013survey, li2008topology, libermann1962quelques}.
 	
 	The quadruplet $(\phi, \xi, \eta, g)$ is said to be a cosymplectic structure if the forms $\eta$ and $\Omega$ are closed, that is, 
 	$$
 	d\eta = 0 \quad \text{and} \quad d\Omega = 0,
 	$$ 
 	where $d$ denotes the exterior derivative (see \cite{li2008topology, libermann1962quelques}).
 	
 	If $(\phi, \xi, \eta, g)$ is a cosymplectic structure  $(\phi, \xi, \eta)$ is normal, then the quadruplet $(\phi, \xi, \eta, g)$ is called a co-K\"ahler structure (see \cite{cappelletti2013survey, li2008topology}). 
 	
 	The normality condition states that the torsion tensor $N$ vanishes, i.e.,
 	\begin{equation}
 		N := [\phi, \phi] + 2\,d\eta \otimes \xi = 0,
 		\label{eq:torsion}
 	\end{equation}
 	where $[\phi, \phi]$ is the Nijenhuis torsion of $\phi$, defined by
 	$$
 	[\phi, \phi](X, Y) = \phi^2[X, Y] + [\phi X, \phi Y] - \phi[\phi X, Y] - \phi[X, \phi Y],
 	$$
 	for any vector fields $X$ and $Y$ on $M$. Note that the normality condition is equivalently to says that obvious induced almost complex structure on $M\times \mathbb{R}$ is integrable (see \cite{sasaki1962differentiable}).
 	
 	In the case $d\eta=\Omega$ the quadruplet $(\phi, \xi, \eta, g)$ is called a contact metric
 	structure. A normal contact structure is called a Sasakian structure (see \cite{blair2010riemannian} for more details and references therein).
 	
 	An almost contact metric manifold $M$, endowed with an almost contact metric structure $(\phi, \xi, \eta, g)$, is called \textit{quasi-Sasakian} if the structure is normal and its second fundamental form $\Omega$ is closed. 
 	
 	The concept of quasi-Sasakian structures was introduced by Blair, as a generalization of Sasakian and co-K\"ahler structures. Blair also provided their first examples (see \cite{blair1966thesis}). The rank of a quasi-Sasakian structure is always odd \cite{blair1966thesis}, it is equal to $1$ if the structure is co-K\"ahler and $2n+1$ is the structure is Sasakian.
 	
 	In $3$-dimensional case, the condition $\eta \wedge \Omega \ne 0$ is satisfied and  ensures that $\Omega$ is nowhere vanishing. Consequently, departing on condition 
 	$$
 	d\eta(\xi,.)=0
 	$$ 
 	which is satisfied by any normal almost contact metric structure, we deduce that there exists a smooth function $\tau : M \to \mathbb{R}$ such that
 	$$
 	d\eta = \tau \,\Omega.
 	$$   
 	Remenber that the Reeb vector field $\xi$ associated with $(\eta,\Omega)$ is uniquely defined by the relations (see \cite{blair2010riemannian})
 	$$
 	\Omega(\xi,\cdot)=0, \;\;  \eta(\xi)=1.
 	$$
 	A direct consequence of the definition is the invariance of the structure under the Reeb flow:
 	\begin{equation} \label{eq:invariance}
 		\mathcal{L}_{\xi}\eta = 0, \;\; 
 		\mathcal{L}_{\xi}\Omega = 0,\;\; 
 		\mathcal{L}_{\xi}(\eta\wedge \Omega) = 0 ,\;\;
 		\mathcal{L}_{\xi}\tau = 0,
 	\end{equation}  
 	where $\mathcal{L}$ denotes the Lie derivative.
 	
 	The kernel of $\Omega$ generated by $\xi$ determines a one-dimensional foliation $\mathcal{F}_\Omega= \mathbb{R}\xi$. The condition $\eta\wedge\Omega\ne0$ implies that $\eta$ does not vanish along $\mathcal{F}_\Omega$, that the hyperplane field $\mathcal{F}_\eta = \ker \eta$ is everywhere transverse to $\mathcal{F}_\Omega$, and that $\Omega$ is nondegenerate on $\mathcal{F}_\eta$. Thus, a quasi-Sasakian structure on a $3$-dimensional manifold induces a canonical splitting
 	$$
 	TM = \mathcal{F}_\Omega \oplus \mathcal{F}_\eta,
 	$$
 	where $\oplus$ is the orthogonal sum, $\mathcal{F}_\Omega$ is an oriented line bundle and $\mathcal{F}_\eta$ is a K\"ahler $2$-plane bundle.

 	\section{Main Results} \label{sec:Result}
 	
 	Let $(M, \phi, \xi, \eta, g)$ be a closed, orientable $3$-manifold equipped with a quasi-Sasakian structure $(\phi, \xi, \eta, g)$ and the second fundamental form $\Omega$. Let $\mathcal{F}_\Omega$ denote the foliation induced by the associated Reeb vector field $\xi$. 
 	
 	A $p$-form $\beta$ is said to be basic with respect to $\mathcal{F}_\Omega$ if it satisfies
 	$$
 	\iota_{\xi} \beta = 0 \quad \text{and} \quad \iota_{\xi} d\beta = 0 .
 	$$
 	We denote by $\Omega^{*}(\mathcal{F}_\Omega)$ the exterior differential complex of basic forms. The cohomology of this complex, denoted by $H_B^{*}(\mathcal{F}_\Omega)$, is called the basic cohomology of the flow. The usual de Rham cohomology of $M$ is denoted by $H^{*}(M)$. Note that, from El Kacimi-Alaoui, Sergiescu and Hector \cite[Th\'eor\`eme 0]{kacimi1985cohomologie}, we remark that the basic cohomology of the foliation $\mathcal{F}_\Omega$ on a compact quasi-Sasakian manifold is always finite-dimensional. 	
 	\begin{theorem}\label{thm:mains}
 	A closed, orientable $3$-dimensional manifold admits a quasi-Sasakian structure if and only if it admits either a Sasakian structure or the structure of a K\"ahler mapping torus.
 	\end{theorem} 	
 	\begin{proof} 
 	Let $(M, \phi, \xi, \eta, g)$ be a closed, orientable $3$-dimensional manifold endowed with a quasi-Sasakian structure $(\phi, \xi, \eta, g)$ and the second fundamental form $\Omega$. The foliation
 		$$
 		\mathcal{F}_\Omega = \ker(\Omega),
 		$$
 		generated by the Reeb vector field $\xi$, is a Killing foliation, and therefore a  Riemannian foliation (see \cite{blair1966thesis, tanno1971quasi}). By the theorem of El~Kacimi-Alaoui, Sergiescu and Hector \cite[Th\'eor\`eme~0]{kacimi1985cohomologie}, the second basic cohomology group $H_B^2(\mathcal{F}_\Omega)$ is either $0$ or $\mathbb{R}$. Moreover, by Molino and Sergiescu \cite[Theorem A]{molino1985deux} and Tondeur \cite[Theorem ~10.17]{tondeur1997geometry}, we have
 		$$
 		H_B^2(\mathcal{F}_\Omega) \cong \mathbb{R}.
 		$$
 		Since $\Omega$ is closed, it defines a basic cohomology class, and hence
 		$$
 		H_B^2(\mathcal{F}_\Omega) = \langle [\Omega] \rangle .
 		$$ 		
 		From Blair \cite{blair1966thesis} and Tanno \cite{tanno1971quasi}, we know that $d\eta(\xi, \cdot) = 0$, so the $2$-form $d\eta$ is basic. Therefore, there exist a constant $r \in \mathbb{R}$ and a basic $1$-form $\alpha$ such that
 		\begin{equation}\label{DifferenETaOmega1}
 		d\eta = r\,\Omega + d\alpha .
 		\end{equation}	
 		\paragraph{Case 1: $r = 0$.}
 		Then $d\eta = d\alpha$, and there exists a $1$-form
 		$$
 		\widetilde{\eta} := \eta - \alpha
 		$$
 		satisfies 
 		$$
 		d\widetilde{\eta} = 0.
 		$$ 		
 		Moreover, since $\iota_{\xi}\alpha = 0$, we have
 		$$
 		\widetilde{\eta} \wedge \Omega = \eta \wedge \Omega,
 		$$
 		which is a volume form. Hence $(M, \Omega, \widetilde{\eta})$ is cosymplectic with Reeb vector field $\xi$. Because $\xi$ is Killing, it follows from \cite{bazzoni2015k} that $(M, \Omega, \widetilde{\eta})$ is a $K$-cosymplectic structure. By Bazzoni and Goertsches \cite[Prop.~2.8]{bazzoni2015k}, $M$ carries a $K$-cosymplectic structure $(\widetilde{\eta}, \xi, \phi, h)$ with $\mathcal{L}_{\xi} h = 0$. Goldberg’s theorem \cite[Proposition~3]{goldberg1969integrability} implies that this structure is normal, hence co-K\"ahler. Therefore, $M$ is a K\"ahler mapping torus as proved by Li in \cite[Theorem~2]{li2008topology}. \\
 		
 		\paragraph{Case 2: $r \neq 0$.}
 		Without loss of generality, we may assume $r = 1$. From (\ref{DifferenETaOmega1}), one has
 		$$
 		\Omega = d(\eta - \alpha),
 		$$
 		Since $\alpha \wedge \Omega = 0$ and letting $\widetilde{\eta} := \eta - \alpha$, we obtain
 		$$
 		\widetilde{\eta} \wedge d\widetilde{\eta} = \eta \wedge \Omega .
 		$$
 		Thus $(M, \widetilde{\eta})$ is a $K$-contact manifold, satisfying
 		$$
 		\iota_{\xi}\widetilde{\eta} = 1, \qquad \iota_{\xi} d\widetilde{\eta} = 0 .
 		$$
 		In dimension $3$, every $K$-contact structure is Sasakian. Hence $(M, \widetilde{\eta})$ is a  Sasakian manifold. 
 		
 		Conversely, as recalled in Section~\ref{sec:basics}, a Sasakian structure is a special case of a quasi-Sasakian structure. Moreover, by Li \cite{li2008topology} (Theorem~2), any K\"ahler mapping torus admits a co-K\"ahler structure, which is also a special case of a quasi-Sasakian structure. This completes the proof.
 	\end{proof}

 	The following corollaries are direct consequences of the Theorem \ref{thm:mains}.
 	
 	\begin{corollary}
 		Any quasi-Sasakian structure $(\phi, \xi, \eta, g)$, with the second fundamental form $\Omega$, on a closed, orientable $3$-dimensional manifold $M$ can be deformed either into a contact structure
 		$$
 		\widetilde{\eta} = \pm \frac{1}{|\Omega|} *\Omega - \alpha,
 		$$
 		or into a co-K\"ahler structure $(\widetilde{\eta}, \Omega)$, with
 		$$
 		\widetilde{\eta} = c\, *\Omega - \alpha,
 		$$
 		for some constant $c \in \mathbb{R}$ and a basic $1$-form $\alpha$, where $*$ is the Hodge operator and $|\Omega|$ represents the norm of $\Omega$ as defined in \cite{blairgoldberg1967topology}.
 	\end{corollary}
 	
 	\begin{proof}
 		From Theorem~\ref{thm:mains}, any quasi-Sasakian structure $(\phi, \xi, \eta, g)$ with the second fundamental form $\Omega$ can be deformed into a the two types of quasi-Sasakian structures, namely Sasakian and co-K\"ahler structures. In particular, there exists a $1$-form $\widetilde{\eta}$ of the form
 		$$
 		\widetilde{\eta} = \eta - \alpha,
 		$$
 		for some basic $1$-form $\alpha$.
 		
 		Using \cite[Proposition 2.12]{blairgoldberg1967topology}, one can further write
 		$$
 		\widetilde{\eta} = \pm \frac{1}{|\Omega|} *\Omega - \alpha,
 		$$
 		where $\Omega = d\widetilde{\eta}$, giving a contact structure.  
 		
 		Alternatively, the structure can be deformed into a co-K\"ahler structure $(\widetilde{\eta}, \Omega)$, with
 		$$
 		\widetilde{\eta} = c\, *\Omega - \alpha,
 		$$
 		for some constant $c \in \mathbb{R}$ (see \cite[Corollary 2.14]{blairgoldberg1967topology}).
 	\end{proof}
 	
 	A natural and important question is to determine, for a given manifold, which of the two situations occurs. Any quasi-Sasakian manifold whose Betti numbers are odd is necessarily a Kähler mapping torus. Indeed, the first Betti number of a Sasakian manifold is always either zero or even (see \cite{blairgoldberg1967topology, tachibana1965harmonic}. Moreover, recall that the first Betti number of a K\"ahler mapping torus is odd (see \cite[Theorem 11]{chinea1993topology}). By Poincar\'e duality, it follows that all Betti numbers of such a manifold are odd. Consequently, a quasi-Sasakian manifold whose first Betti number is even cannot be a K\"ahler mapping torus and must therefore be Sasakian.
 	
 	Therefore, we now provide a classification of $3$-dimensional Sasakian and co-Kähler manifolds.
 	
 	\begin{proposition}
 		\label{thm:cokha}
 		Let $M$ be a closed, orientable $3$-dimensional manifold. Then $M$ admits the structure of a K\"ahler mapping torus if and only if $M$ is diffeomorphic to one of the following:
 		\begin{enumerate}
 			\item $\mathbb{S}^2 \times \mathbb{S}^1$; 			
 			\item a torus bundle $\mathbb{T}^2_A$ over $\mathbb{S}^1$ with monodromy
 			$$
 			A \in \left\{
 			\begin{pmatrix} 1 & 0 \\ 0 & 1 \end{pmatrix},
 			\begin{pmatrix} 0 & 1 \\ -1 & 0 \end{pmatrix},
 			\begin{pmatrix} -1 & 0 \\ 0 & -1 \end{pmatrix},
 			\begin{pmatrix} 0 & -1 \\ 1 & 1 \end{pmatrix},
 			\begin{pmatrix} -1 & -1 \\ 1 & 0 \end{pmatrix}
 			\right\};
 			$$
 			
 			\item a quotient $(\mathbb{H}^2 \times \mathbb{R}) / \Xi$, where 
 			$\Xi \subset \mathrm{Isom}(\mathbb{H}^2 \times \mathbb{R})$ is a discrete subgroup acting freely.
 		\end{enumerate}  		
 	\end{proposition}   		
 	\begin{proof}
 	Assume that $M$ is a K\"ahler mapping torus. Then $M \cong \Sigma_\varphi$, where   $\Sigma_\varphi$ a mapping torus of $\varphi$ defined as
 	$$
 	\Sigma_{\varphi} = \frac{\Sigma \times [0,1]}{(x,0) \sim (\varphi(x),1)}.
 	$$
 	and $(\Sigma, J, h)$ is a compact K\"ahler surface and $\varphi$ is a Hermitian isometry (see Li in \cite{li2008topology} for more details).
 		Let $\mathrm{Isom}(\Sigma, h)$ denote the group of Hermitian isometries of $\Sigma$ and
 		$\mathrm{Isom}_0(\Sigma,h)$ its identity component.
 		By \cite[Theorem~6.6]{bazzoni2015k}, the diffeomorphism $\varphi$ is either
 		isotopic to the identity or has finite order in
 		$$
 		\mathrm{Isom}(\Sigma,h)/\mathrm{Isom}_0(\Sigma,h).
 		$$
 		In particular, $\varphi$ is periodic.
 		By \cite[Theorem~5.4]{scott1983geometries}, the manifold $M$ is a Seifert fibered space
 		over a two-dimensional orbifold $B$ with vanishing Euler number. Moreover, the Thruston geometries
 		$\mathbb{S}^3$, $\mathrm{Nil}$, and $\widetilde{\mathrm{SL}_2(\mathbb{R})}$ do not occur
 		in this setting. Hence, the classification depends on the orbifold Euler characteristic
 		$\chi^{\mathrm{orb}}(B)$.\\
 		
 		\paragraph{Case 1: $\chi^{\mathrm{orb}}(B) > 0$.}  In this case, $M$ admits $\mathbb{S}^2 \times \mathbb{R}$ geometry. By \cite[Proposition~10.3.36]{martelli2016introduction}, $M$ is diffeomorphic to one of
 		$\mathbb{S}^2 \times \mathbb{S}^1$, the orientable $\mathbb{S}^1$-bundle over
 		$\mathbb{RP}^2$, or a Seifert manifold of type $(S^2;(p,q),(p,-q))$.
 		All manifolds of the latter type are diffeomorphic to $\mathbb{S}^2 \times \mathbb{S}^1$.
 		Moreover, by \cite[Exercise~10.3.37]{martelli2016introduction}, the orientable
 		$\mathbb{S}^1$-bundle over $\mathbb{RP}^2$ is diffeomorphic to
 		$\mathbb{RP}^3 \# \mathbb{RP}^3$, which has vanishing first Betti number and is therefore
 		excluded. Hence, the only possibility is $M \cong \mathbb{S}^2 \times \mathbb{S}^1$.\\

 		\paragraph{Case 2: $\chi^{\mathrm{orb}}(B) = 0$.} Then $M$ admits Euclidean geometry and carries a flat metric
 		\cite[Theorem~12.3.1]{martelli2016introduction}.
 		From the classification of closed flat $3$-manifolds
 		\cite{hantzsche1935dreidimensionale, Wolf2011Spaces},
 		the first homology group $H_1(M,\mathbb{Z})$ is isomorphic to one of
 		$\mathbb{Z}^3$, $\mathbb{Z}\oplus\mathbb{Z}_2^2$, $\mathbb{Z}\oplus\mathbb{Z}_3$,
 		$\mathbb{Z}\oplus\mathbb{Z}_2$, $\mathbb{Z}$, or $\mathbb{Z}_4^2$. The Hantzsche-Wendt manifold, corresponding to $\mathbb{Z}_4^2$, has $b_1(M)=0$ and is
 		thus excluded. All remaining cases have positive first Betti number and, by
 		\cite[Theorem~3.2]{marrero1998new}, $M$ is diffeomorphic to one of
 		$M_1(1)$, $M_2(1)$, $M'_1(1)$, $M'_2(1)$, or $\mathbb{T}^3$.
 		As shown in \cite{marrero1998new}, these manifolds are precisely the torus bundles
 		$\mathbb{T}^2_A$ with
 		$$
 		A = \begin{pmatrix}0 & 1 \\ -1 & 0\end{pmatrix},\;
 		\begin{pmatrix}-1 & 0 \\ 0 & -1\end{pmatrix},\;
 		\begin{pmatrix}0 & -1 \\ 1 & 1\end{pmatrix},\;
 		\begin{pmatrix}-1 & -1 \\ 1 & 0\end{pmatrix},
 		$$
 		and the identity matrix.\\
 		
 		\paragraph{Case 3: $\chi^{\mathrm{orb}}(B) < 0$.} 
 		In this case, $M$ admits $\mathbb{H}^2 \times \mathbb{R}$ geometry and hence is
 		diffeomorphic to a quotient
 		$$
 		M \cong (\mathbb{H}^2 \times \mathbb{R})/\Xi,
 		$$
 		where $\Xi \subset \mathrm{Isom}(\mathbb{H}^2 \times \mathbb{R})$ is a discrete subgroup
 		acting freely \cite{scott1983geometries}.
 		
 		Conversely, the manifolds in Proposition~\ref{thm:cokha} are precisely the Seifert bundles with
 		vanishing Euler number (see \cite{scott1983geometries}). Denote by $\mathcal{F}$ the
 		circle foliation whose leaves are the fibers of the Seifert bundle. By
 		\cite{epstein1972periodic}, there exists an $\mathbb{S}^1$-action whose orbits are exactly
 		the leaves of $\mathcal{F}$. By Wadsley’s theorem \cite{wadsley1975geodesic}, there exist a
 		Riemannian metric $g$ and a unit Killing vector field $R$ such that the trajectories of
 		$R$ coincide with the leaves of $\mathcal{F}$.
 		
 		Now define
 		$$
 		\lambda := \iota_R g,
 		$$
 		let $v_g$ be the Riemannian volume form associated with $g$, and define the $2$-form
 		$$
 		\omega := \iota_R v_g.
 		$$
 		Since $R$ is Killing, we have
 		$$
 		d\omega = d(\iota_R v_g) = \mathcal{L}_R v_g = 0.
 		$$
 		Moreover, since $\iota_R(\lambda \wedge v_g)=0$, it follows that
 		$$
 		v_g = \lambda \wedge \omega.
 		$$ 		
 		A straightforward computation shows that
 		$$
 		d\lambda(R,X)=0.
 		$$
 		Therefore
 		$$
 		[d\lambda] \in H_B^2(\mathcal{F}).
 		$$
 		From \cite{ prieto2001euler, saralegui1985euler}, we deduce that
 		$$
 		[d\lambda] = e_{\mathcal{F}},
 		$$
 		where $e_{\mathcal{F}}$ denotes the Euler class of the Seifert bundle defined by
 		$\mathcal{F}$. By assumption, this class vanishes. Therefore, there exists a basic
 		$1$-form $\alpha$ such that
 		$$
 		d\lambda = d\alpha.
 		$$
 		By letting $\widetilde{\lambda} := \lambda - \alpha$, we have $d\widetilde{\lambda}=0$, and since $\iota_R \alpha = 0$, we obtain
 		$$
 		\widetilde{\lambda} \wedge \omega = \lambda \wedge \omega = v_g,
 		$$
 		which is a volume form. Hence $(M,\omega,\widetilde{\lambda})$ is a cosymplectic structure
 		with Reeb vector field $R$. Since $R$ is Killing, it follows from \cite{bazzoni2015k} that
 		$(M,\omega,\widetilde{\lambda})$ is a $K$-cosymplectic manifold. Similarly, as in the proof
 		of Theorem~\ref{thm:mains}, we deduce that $M$ admits a co-K\"ahler structure. Therefore,
 		$M$ is a K\"ahler mapping torus (see \cite[Theorem~2]{li2008topology}). This concludes the proof.
 	\end{proof}
 	In the case of Sasakian manifolds, an analogous result can be obtained as follows.
 	
 	\begin{proposition}\label{thm:sasak}
 		Let $M$ be a closed, oriented $3$-manifold. Then $M$ admits a Sasakian structure if and only if $M$ is diffeomorphic to one of the following:
 		$$
 		\mathrm{Nil}/\Gamma, \quad 
 		\mathbb{S}^3/\Gamma, \quad 
 		\widetilde{\mathrm{SL}(2,\mathbb{R})}/\Gamma,
 		$$
 		where $\Gamma$ is a discrete subgroup of the identity component of the corresponding isometry group.
 	\end{proposition} 	
 	\begin{proof} 
 	Let $(\phi, \xi, \eta, g)$ be a Sasakian structure on $M$, with the second fundamental form
 		$d\eta = \Omega$. Then $\eta$ is a $K$-contact form (see Banyaga in \cite{banyaga1990note}).
 		By Theorem of Monna \cite[Theorem, p.~86]{monna1984feuilletages}, the
 		one-dimensional foliation $\mathcal{F}_\Omega$ defines a Seifert fibration on $M$,
 		and this foliation does not admit any complementary Riemannian foliation.
 		By \cite[Theorem~B]{saralegui1985euler}, we deduce that the Euler class
 		$e_{\mathcal{F}}$ is nonzero. Then, from \cite{scott1983geometries}, $M$ admits one
 		of the following Thruston geometries:
 		$$
 		\mathbb{S}^3, \quad \mathrm{Nil}, \quad \widetilde{\mathrm{SL}_2(\mathbb{R})}.
 		$$ 
 		Conversely, the manifolds in Proposition~\ref{thm:sasak} are precisely the Seifert bundles with
 		nonzero Euler number (see \cite{scott1983geometries}). Using the same arguments as in
 		the sufficiency part of the proof of Proposition~\ref{thm:cokha}, we obtain that
 		$$
 		[d\lambda] = e_{\mathcal{F}},
 		$$
 		where $e_{\mathcal{F}}$ denotes the Euler class of the Seifert bundle defined by the
 		circle foliation $\mathcal{F}$, and $\lambda$ is the $1$-form defined in the proof of
 		Proposition~\ref{thm:cokha}. Since $e_{\mathcal{F}} \neq 0$, there exists a basic
 		$1$-form $\alpha$ such that
 		$$
 		\omega = d(\lambda - \alpha),
 		$$
 		where $\omega$ is defined as in the proof of Proposition~\ref{thm:cokha}. Set
 		$ 
 		\widetilde{\lambda} := \lambda - \alpha.
 		$ 
 		Since $\alpha \wedge \omega = 0$, we obtain
 		$$
 		\widetilde{\lambda} \wedge d\widetilde{\lambda}
 		= \lambda \wedge \omega = v_g,
 		$$
 		where $v_g$ is the Riemannian volume form. Thus $(M, \widetilde{\lambda})$ is a
 		$K$-contact manifold satisfying
 		$$
 		\iota_R \widetilde{\lambda} = 1, \qquad
 		\iota_R d\widetilde{\lambda} = 0.
 		$$
 		In dimension $3$, every $K$-contact structure is Sasakian. Hence $M$ admits a
 		Sasakian structure. This concludes the proof.
 	\end{proof} 	
 	From Theorem~\ref{thm:mains}, Proposition~\ref{thm:cokha}, and
 	Proposition~\ref{thm:sasak}, we deduce the following.
 	
 	\begin{corollary}\label{cor:classification}
 		A closed, orientable $3$-dimensional manifold admits a quasi-Sasakian structure if
 		and only if it is diffeomorphic to one of the manifolds listed in
 		Proposition~\ref{thm:sasak} or in Proposition~\ref{thm:cokha}.
 	\end{corollary}
 	
 	\begin{proof}
 		The result follows directly from the previous statements.
 	\end{proof}
 	
 	To conclude this section, we present a more general result for quasi-Sasakian manifolds.
 	
 	\begin{proposition}\label{thm:main}
 		Let $M$ be a closed, orientable $(2n+1)$-dimensional manifold admitting a quasi-Sasakian
 		structure $(\phi,\xi,\eta,g)$ with the second fundamental form $\Omega$ such that
 		$$
 		H_B^2(\mathcal{F}_\Omega) \cong \mathbb{R}.
 		$$
 		Then $M$ admits either a $K$-cosymplectic structure or a $K$-contact structure.
 	\end{proposition} 	
 	\begin{proof}
 		By assumption $H_B^2(\mathcal{F}_\Omega) \cong \mathbb{R}$, there exists
 		$c \in \mathbb{R}$ such that
 		$$
 		[\Omega] = c [d\eta] \in H_B^2(\mathcal{F}_\Omega).
 		$$
 		Using the same arguments as in the proof of the first part of Theorem~\ref{thm:mains}, we conclude that $(M,\Omega,\widetilde{\eta})$ is either
 		$K$-cosymplectic or $K$-contact, with Reeb vector field $\xi$.
 	\end{proof}
 	
 Recently Di Pinto and Dileo in \cite{DiPinto1} and \cite{DiPinto2} have introduce a new class of almost contact metric manifolds, called anti-quasi-Sasakian (aqS) which states that an anti-quasi-Sasakian
 manifold is an almost contact metric manifold such that
 $$
 d\Omega=0 \;\;\mbox{and}\;\;  N =   2\,d\eta \otimes \xi.
 $$ 
 In the single authored paper \cite{DiPinto1}, Di Pinto prove that every compact aqS manifold has nonvanishing second Betti number. Therefore we have the following.
 \begin{proposition}
 	Let $M$ be a closed, orientable $3$-dimensional manifold admitting an anti-quasi-Sasakian
 structure $(\phi,\xi,\eta,g)$ with the second fundamental form $\Omega$. Then $M$ is either  Sasakian  or K\"ahler mapping torus.  
 \end{proposition}		
 \begin{proof}
 Consider an anti-quasi Sasakian structure $(\phi,\xi,\eta,g)$ on $M$. By \cite[Proposition 1.13]{DiPinto2}, the Reeb vector field $\xi$ is Killing, then the foliation given $\mathcal{F}_{\Omega}$ is Killing. By the proof of Theorem \ref{thm:mains},
 $$
 H_B^2(\mathcal{F}_\Omega) \cong \mathbb{R}. 
 $$
From Proposition 1.3 in \cite{DiPinto2}, $d\eta(\xi, \cdot)=0$, and the fact that $d\Omega=0$, then there exist  a constant $r \in \mathbb{R}$ and a basic $1$-form $\alpha$ such that
$$
	d\eta = r\,\Omega + d\alpha,
$$
as per the proof of Theorem \ref{thm:mains} above, which completes the proof.
 \end{proof}

 	\section*{Acknowledgments}	
 	
 	The authors are grateful to A. Banyaga (PennState University) for invaluable discussions and support.

 	\end{document}